\documentclass[a4paper,oneside]{amsart}
\usepackage{graphicx,amsfonts,amsmath,amssymb}
\usepackage{hyperref,amssymb,amsmath}
\usepackage{gauss}\newmatrix{[}{]}{q}
\usepackage[table]{xcolor}
\usepackage{float}
\newtheorem{thm}{Theorem}[section]
\newtheorem{cor}[thm]{Corollary}
\newtheorem{prob}[thm]{Problem}

\newtheorem{que}[thm]{Question}

\newtheorem{lem}[thm]{Lemma}
\newtheorem{prop}[thm]{Proposition}
\theoremstyle{definition}
\newtheorem{defn}[thm]{Definition}
\theoremstyle{remark}
\newtheorem{rem}[thm]{Remark}
\numberwithin{equation}{section}

\definecolor{g}{rgb}{0.4, 1.0, 0.0}
\definecolor{p}{rgb}{1.0, 0.74, 0.85}
\definecolor{lg}{rgb}{0.86, 0.82, 1.0}
\definecolor{y}{rgb}{1.0, 1.0, 0.0}
\allowdisplaybreaks[4]
\begin{document}

\title[Character sums]{Non-abelian finite groups whose character sums are invariant but are not Cayley isomorphism}%
\author[A. Abdollahi]{Alireza Abdollahi}%
\address{Department of Mathematics, University of Isfahan, Isfahan 81746-73441, Iran; and School of Mathematics, Institute for Research in Fundamental Sciences (IPM), P.O. Box 19395-5746, Tehran, Iran}%
\email{a.abdollahi@math.ui.ac.ir}%
\author[M. Zallaghi]{Maysam Zallaghi}%
\address{Department of Mathematics, University of Isfahan, Isfahan 81746-73441, Iran}%
\email{m.zallaghi@sci.ui.ac.ir \\ maysam.zallaghi@gmail.com}%

\thanks{}%

\begin{abstract}
Let $G$ be a group and $S$ an inverse closed subset of $G\setminus \{1\}$. By a Cayley graph $Cay(G,S)$ we mean the graph whose vertex set is the set of elements of $G$ and two vertices $x$ and $y$ are adjacent if $x^{-1}y\in S$.
A group $G$ is called a CI-group if $Cay(G,S)\cong Cay(G,T)$ for some inverse closed subsets $S$ and $T$ of $G\setminus \{1\}$,
then $S^\alpha=T$ for some automorphism $\alpha$ of $G$. A finite group $G$ is called a BI-group if $Cay(G,S)\cong Cay(G,T)$ for some inverse closed subsets $S$ and $T$ of $G\setminus \{1\}$,
then $M_\nu^S=M_\nu^T$ for all positive integers $\nu$, where $M_\nu^S$ denotes the set $\big\{\sum_{s\in S}\chi(s) | \chi(1)=\nu, \chi \text{ is a complex irreducible character of } G \big\}$. It was asked by L\'aszl\'o Babai [\textit{J. Combin. Theory Ser. B}, {\bf  27} (1979) 180-189] if every finite group is a BI-group; various examples of  finite
non BI-groups are presented in [\textit{Comm. Algebra}, {\bf  43} (12) (2015) 5159-5167]. It is noted in the latter paper that every finite CI-group is a BI-group and all abelian finite groups are BI-groups. However it is known that there are finite abelian non CI-groups. Existence of a finite non-abelian BI-group which is not a CI-group is the main question  which we study here.  
We find two non-abelian BI-groups of orders $20$ and $42$ which are not CI-groups. 
We  also list all BI-groups of orders up to $30$.
\end{abstract}
\maketitle

\section{\bf Introduction and Results}
In this paper all graphs are finite and simple so graphs has a finite number of vertices and they are undirected without multiple edges and 
has no loop. 
Let $G$ be a group and $S$ an inverse closed subset of $G^*:=G\setminus \{1\}$. By a Cayley graph $Cay(G,S)$ we mean the graph whose vertex set is the set of elements of $G$ and two vertices $x$ and $y$ are adjacent if $x^{-1}y\in S$.

There is a  famous relation between complex irreducible characters of a finite group $G$ and spectra of Cayley graphs $Cay(G,S)$  (see \cite{B}, see  Theorem \ref{Babai}, below). According to the latter result a character sum $\sum_{s\in S} \chi(s)$ on $S$ for any complex irreducible character $\chi$ of $G$ is equal to a sum of $\chi(1)$ eigenvalues of $Cay(G,S)$. 

L\'aszl\'o Babai \cite{B} has proposed whether the set of character sums of all complex irreducible characters of the same degree is an invariant of $Cay(G,S)$:

\begin{prob}\cite[Problem 3.3]{B}\label{prob} Let $G$ be a finite group, let $\Gamma$ be a Cayley graph $Cay(G,S)$, $\nu$ a positive integer and
\begin{displaymath}
\mu_i=\sum_{s\in S}\chi_i(s) \;\;\; (i=1,\dots,h),
\end{displaymath}
where  $\chi_1,\dots,\chi_h$ are all irreducible characters of $G$. 
Is the set $ M_\nu^S=\{\mu_i \;\vert\; \chi_i(1)=\nu\} $ an \textit{invariant} of $ \Gamma$? (Thus, does $Cay(G,S)\cong Cay(G,S')$ imply $M_\nu^S=M_\nu^{S'}$?)
\end{prob}
\begin{rem}
Babai in \cite[Problem 3.3]{B} mentions  Problem \ref{prob} for all Cayley {\em digraphs} and according to his definition of a Cayley digraph \cite[p. 182]{B} a Cayley graph can be considered as a Cayley digraph.  
\end{rem}

As we mentioned above, here we only study the simple undirected Cayley graphs.
In \cite{AM}  various examples of finite groups not satisfying the property mentioned in Problem \ref{prob} are presented.    
So we need to name the class of groups satisfying the property.

\begin{defn}
A Cayley graph $Cay(G,S)$ is called a \textit{Babai Isomorphism} graph or a \textit{BI-graph}   if  $Cay(G,S)\cong Cay(G,T)$ for some inverse closed subset $T$ of $G^*$ then $M_\nu^S=M_\nu^T$ for all complex irreducible characters of degree $\nu$ of $G$. A group $G$ is called a \textit{BI-group} if all  Cayley graphs of $G$ are BI-graphs. 
\end{defn}

A famous class of Cayley graphs is the class of Cayley isomorphism graphs which is an still active research area (see \cite{B, CRS, Dobson1, Lov, So1}). 

\begin{defn}
A Cayley graph $Cay(G,S)$ is called a Cayley Isomorphism graph or CI-graph if  $Cay(G,S)\cong Cay(G,T)$ for some inverse closed subset $T$ of $G^*$ then $S^\alpha=T$ for some automorphism $\alpha$ of $G$. A group $G$ is called a CI-group if $Cay(G,S)\cong Cay(G,T)$ for some inverse closed subsets $S$ and $T$ of $G^*$,
then $S^\alpha=T$ for some automorphism $\alpha$ of $G$.
\end{defn}

Every Cayley isomorphism graph is a BI-graph \cite[Proposition 2.5]{AM}. Every finite abelian group is a BI-group \cite[Proposition 2.2]{AM} but not every abelian group is a CI-group \cite[Theorem 2.11]{Als}. 
We give the first example of non-abelian finite BI-groups which are not CI-groups.

 \begin{thm}\label{F5}
The group $\mathcal{G}=\langle a,b|a^5=b^4=1, a^b=a^3\rangle$ is the smallest non-abelian BI-group which is not a CI-group.
\end{thm}

\begin{thm}\label{F7}
The group $\mathcal{H}=\langle a,b|a^7=b^6=1, a^b=a^3\rangle$ is a non-abelian BI-group which is not a CI-group.
\end{thm}

The question of which finite groups are BI-groups has been proposed in \cite{AM}:
\begin{que} \cite[Question 3.9]{AM}
Which finite groups are BI-groups?
\end{que}

Here we determine all finite BI-groups of order at most $30$. We denote by $C_n$, $S_n$, $A_n$, $D_n$ and $Q_n$
the cyclic group of order $n$, the symmetric group of degree $n$, the alternating group of degree $n$, 
the dihedral group of order $n$ and the generalized quaternion group of order $n$, respectively. If $H$ and $K$ are two groups, we denote by
$H\ltimes K$ the semidirect product of $K$ by $H$ and since in general using the notation $H \ltimes K$ for a group does not uniquely determine the group, we mention its GAP {\tt IdSmallGroup} (see \cite{GAP} for the usage of the latter code). 
\begin{thm}\label{BI-groups}
Let $G$ be a finite non-abelian group of order at most $30$. Then, $G$ is a BI-group if and only if $G$ is isomorphic to one of the following groups:\\
$S_3$, $Q_8$, $D_{10}$,\\
$C_4\ltimes C_3=\langle a,b\; |\; a^4=b^3=b^ab=1 \rangle$ $({\tt IdSmallGroup} [12,1])$,\\
$A_4$, $D_{14}$, $D_{18}$,\\
$C_2\ltimes (C_3\times C_3)=\langle a,b,c \;|\; a^2=b^3=c^3=[b,c]=b^a b=c^a c=1 \rangle$ $({\tt IdSmallGroup } [18,4])$,\\
$C_4\ltimes C_5=\langle a,b\; |\; a^4=b^5=b^ab=1 \rangle$ $({\tt IdSmallGroup } [20,1])$,\\
$C_4\ltimes C_5=\langle a,b\; |\; a^4=b^5=b^ab^2=1 \rangle$ $({\tt IdSmallGroup } [20,3])$,\\
$C_3\ltimes C_7=\langle a,b\; |\; a^3=b^7=b^ab^3=1 \rangle$ $({\tt IdSmallGroup } [21,1])$,\\
$D_{22}$, $C_8\ltimes C_3=\langle a,b\; |\; a^8=b^3=b^ab=1 \rangle$ $({\tt IdSmallGroup } [24,1])$,\\
$D_{26}$, $C_4\ltimes C_7=\langle a,b\; |\; a^4=b^7=b^ab=1 \rangle$ $({\tt IdSmallGroup } [28,1])$,\\
$C_5\times S_3$, $C_3\times D_{10}$, $D_{30}$.
\end{thm}

\section{\bf Preliminaries}

The following result of Babai \cite{B} gives  a useful relation between spectra of Cayley graphs and the irreducible characters of the underlying group.
\begin{thm}\cite[Theorem 3.1]{B}\label{Babai}
Let $G$ be a finite group whose all complex irreducible characters are denoted by $\chi_1,\dots,\chi_h$ and the degree $\chi_i(1)$ of $\chi_i$ is 
denoted by $n_i$ for each $i=1,\dots,h$. Suppose $S$ is an inverse closed subset of $G^*$.
Then, the spectrum of the Cayley graph $Cay(G,S)$ can be arranged as 
$$spec(Cay(G,S))=\{\lambda_{ijk}| i=1,\dots ,h; j,k=1,\dots,n_i\},$$
where $\lambda_{ij1}=\dots=\lambda_{ijn_i}$, and this common value will be denoted by $\lambda_{ij}$. Also:
$$\lambda_{i1}^t+\dots +\lambda_{in_i}^t=\sum_{s_1,\dots ,s_t\in S} \chi_i (s_1\cdots s_t),$$
for any positive integer $t$.
\end{thm}

It was shown in \cite{AM} that all CI-groups are BI-groups.
\begin{thm}\cite[Propositions 2.6]{AM}\label{BICI}
Every finite CI-group is a BI-group.
\end{thm}

For a finite group $G$, $Irr(G)$ denotes the set of all complex irreducible characters of $G$; and $1_G$ denotes the principal character of $G$
 
\begin{lem}
Let $G$ be a finite group and $S,T$ be inverse closed subsets of $G^*$. If $M_\nu ^S=M_\nu ^T$, then $M_\nu ^{S^c}=M_\nu ^{T^c}$, where $S^c=G^*\setminus S$.
\end{lem}
\begin{proof}
 We may write
$$\sum_{s\in S^c}\chi(s)=\sum_{g\in G^*}\chi(g)-\sum_{s\in S}\chi(s)$$
for any $\chi \in Irr(G)$.
By \cite[Corollary 2.14]{Isaacs}
\begin{displaymath}
\sum_{g\in G^*} \chi(g)=\left\{ {\begin{array}{cc}
|G|-1 & \chi=1_G\\\\
-\chi(1) & \chi\neq 1_G
\end{array}}\right..
\end{displaymath}
Therefore
\begin{displaymath}
\sum_{s\in S^c} \chi(s)=\left\{ \begin{array}{rcl}
|G|-1-|S| & \chi=1_G\\
-\chi(1) -\sum_{s\in S} \chi(s) & \chi\neq 1_G
\end{array}.\right.
\end{displaymath}
It follows that  
\begin{equation}\label{eq}
\big\{ \sum_{s\in S^c} \chi(s) \;|\;  \chi \in Irr(G), \chi(1)=\nu, \chi\neq 1_G \big\}= \big \{ \sum_{t\in T^c} \chi(t) \;|\;  \chi \in Irr(G), \chi(1)=\nu, \chi\neq 1_G \big \}.
\end{equation}
 
Therefore $M_\nu ^{S^c} = M_\nu ^{T^c}$ for all $\nu>1$, since the degree of $1_G$ is $1$.

Now assume that $\nu=1$. 

By Theorem \ref{Babai}, the set $M_1^S$ consists of some eigenvalues of $Cay(G,S)$. Also $|S|=\sum_{s\in S} 1_G(s) \in M_1^S$ is the largest eigenvalues of $Cay(G,S)$ \cite[Proposition 1.1.2]{CRS}. Therefore $|S|$ is the maximum of $M_1^S$ and so $|S|=|T|$, since $M_1^S=M_1^T$. It follows from the set equality \eqref{eq} that $M_1^{S^c}=M_1^{T^c}$. This completes the proof.    
\end{proof}

\begin{cor}\label{cor1}
Let $G$ be a finite group and $S$ be an inverse closed subset of $G^*$. Then, $Cay(G,S)$ is a BI-graph if and only if $Cay(G,S^c)$ is a BI-graph.
\end{cor}

\begin{prop}\label{red-BI}
Let $G$ be a finite group. Then, $G$ is a BI-group if and only if $Cay(G,S)$ is a BI-graph for all inverse closed subsets $S$ of $G^*$ such that $G=\langle S \rangle$ and $\frac{|G|}{2}-1\leq |S| \leq |G|-1$.
\end{prop}
\begin{proof}
It follows from Corollary \ref{cor1} and the fact that the complement of a disconnected graph is a connected graph, that a  $G$ is a BI-group if and only if $Cay(G,S)$ is a BI-graph for all inverse closed subsets $S$ of $G^*$ such that $G=\langle S \rangle$. Now if $|S|<\frac{|G|}{2}-1$, then $|S^c|>\frac{|G|}{2}$ and so $G=S^c S^c$. In particular  $G=\langle S^c\rangle$. This completes the proof.
\end{proof}

With the aid of the following codes  written in  GAP \cite{GAP} one can check whether a given finite group is a BI-group or not.\\
The function {\tt M} determines whether for  two given inverse closed subsets {\tt S} and {\tt T} without non-trivial element of a group {\tt g} and a given positive integer {\tt v}, $M_{\tt v}^{\tt S}=M_{\tt v}^{\tt T}$ or not. The function {\tt BIn} determines for each character degree $v$ of a given finite group {\tt g} all pairs $(S,T)$ of inverse closed subsets of a given size {\tt n} of {\tt g} such that   $Cay({\tt g},S)\cong Cay({\tt g},T)$ and $M_v^S \neq M_v^T$. The function {\tt BI} determines whether or not a given finite group {\tt g} is a BI-group by giving all pairs $(S,T)$ and positive integers $v$ such that $Cay({\tt g},S)\cong Cay({\tt g},T)$ and $M_v^S \neq M_v^T$.  
\begin{verbatim}
LoadPackage("Grape");

M:=function(g,S,T,v)
local S1,T1,Xv,MSv,MTv,C,D;
D:=CharacterDegrees(g);
C:=ConjugacyClasses(g);
S1:=List(S,j->Filtered([1..Size(C)],i->(j in C[i])=true)[1]);
T1:=List(T,j->Filtered([1..Size(C)],i->(j in C[i])=true)[1]);
Xv:=Filtered(Irr(g),i->i[1]=v);
MSv:=Set(List(Xv,i->Sum(S1,j->i[j])));
MTv:=Set(List(Xv,i->Sum(T1,j->i[j])));
return MSv=MTv;
end;

BIn:=function(g,n)
local SS,TT,GG,NG,NN;
SS:=Filtered(Combinations(Difference(Elements(g),[One(g)]),n),i->
		Set(i,k->k^-1)=i);
TT:=Set(Filtered(SS, i->GroupWithGenerators(i)=g));
GG:=Filtered(Combinations( TT , 2 ),i->IsIsomorphicGraph
		(CayleyGraph(g,i[1]),CayleyGraph(g,i[2])));
NG:=List(Set( Irr( g ), DegreeOfCharacter ),i->[i,Filtered(GG,j->
		M(g,j[1],j[2],i)=false)]);
return NG;
end;

BI:=function(g)
return List([(Size(g)/2)-1..Size(g)-1],i->BIn(g,i));
end;
\end{verbatim} \label{caodes}



\section{\bf Proof of Theorem \ref{F5}}

\begin{defn}\label{spe}
\begin{enumerate}
\item Let $\Gamma$ be a graph. We denote  by $spec(\Gamma)=(\mu_1^{l_1},\mu_2^{l_2},\dots,\mu_k^{l_k})$ the spectrum of $\Gamma$ \cite{CRS}, where $\mu_i$'s are all distinct (real) eigenvalues of $\Gamma$ and $l_i$ is the multiplicity of $\mu_i$ for $i=1,\dots,k$.
\item Let $G$ be a finite group. We denote by $S_k$ the set of all elements of order $k$ in $S\subset G$. So, if $\{l_1,\dots, l_r\}$ be the set of  orders of elements of  $S$, then:
$$S=S_{l_1}\cup \dots \cup S_{l_r}.$$
\end{enumerate}
\end{defn}



\begin{thm}\label{g20}
Let $\mathcal{G}=\langle a,b|a^5=b^4=1, a^b=a^3\rangle$, $S\subset \mathcal{G}^*$, $S^{-1}=S$, $\mathcal{G}=\langle S  \rangle$ and $\Gamma=Cay(\mathcal{G},S)$. Then the spectrum of $\Gamma$ is either
$$spec(\Gamma)=(\mu_1^1=|S|,\mu_2^{4k+1},\mu_3^{4k'+2},\mu_{4_1}^{4l_1},\dots,\mu_{4_e}^{4l_e})$$ 
or
$$spec(\Gamma)=(\mu_1^1=|S|,\theta^{4k''+3},\mu_{4_1}^{4l'_1},\dots,\mu_{4_f}^{4l'_f}),$$
where
$\mu_2=|S_2|-|S_4|+|S_5|$, $\mu_3=|S_5|-|S_2|$ with $\mu_2\neq \mu_3$, $\theta=|S_2|-|S_4|+|S_5|=|S_5|-|S_2|$, $0\leq k,k',k''\leq 4$, $0\leq l_i,l'_j\leq 4$ and $0 \leq e,f\leq 4$. In any case $|S|$ is the unique largest eigenvalue of $\Gamma$. 
\end{thm}

\begin{proof}
The group $\mathcal{G}$ has $5$ conjugacy classes: $\mathcal{G}_1$, $\mathcal{G}_2$ and $\mathcal{G}_5$ are the conjugacy classes consisting of all elements of $\mathcal{G}$ of order $1$, $2$ and $5$, respectively; there are two conjugacy classes denoted by $\mathcal{G}_{4_1}$ and $\mathcal{G}_{4_2}$ whose elements are of order $4$ and if $x$ is an element of order $4$, then $x$ and $x^{-1}$ are not conjugate.
Let $S_{4_i}=S \cap \mathcal{G}_{4_i}$ for $i=1,2$; so $S_{4_1}^{-1}=S_{4_2}$. According to Definition \ref{spe}:
$$|S|=|S_2|+|S_{4_1}|+|S_{4_2}|+|S_5|,$$
and note that $S_4=S_{4_1}\cup S_{4_2}$.
The character table of $\mathcal{G}$ is:
\begin{table}[H]
\centering
\begin{tabular}{c||ccccc} 
Class&$\mathcal{G}_1$&$\mathcal{G}_2$&$\mathcal{G}_{4_1}$&$\mathcal{G}_{4_2}$&$\mathcal{G}_5$\\\hline
Size&$1$&$5$&$5$&$5$&$4$\\\hline
$\chi_1$&$1$&$1$&$1$&$1$&$1$\\
$\chi_2$&$1$&$1$&$-1$&$-1$&$1$\\
$\chi_3$&$1$&$-1$&$i$&$-i$&$1$\\
$\chi_4$&$1$&$-1$&$-i$&$i$&$1$\\
$\psi$&$4$&$0$&$0$&$0$&$-1$\\
\end{tabular}
\caption{\label{table:F}\small{Character table of the group $\mathcal{G}$}}
\end{table}

Now by Theorem \ref{Babai}, all eigenvalues of $\Gamma$ corresponding to linear characters $\chi_1,\dots,\chi_4$ of $\mathcal{G}$ are as follows:
\begin{align*}
\lambda_1 &=\sum_{s\in S}\chi_1 (s)=|S|=|S_2|+|S_{4_1}|+|S_{4_2}|+|S_5| = |S_2|+|S_4|+|S_5|\\
\lambda_2 &=\sum_{s\in S}\chi_2 (s)=|S_2|-|S_{4_1}|-|S_{4_2}|+|S_5|=|S_2|-|S_4|+|S_5|\\
\lambda_3 &=\sum_{s\in S}\chi_3 (s)=-|S_2|+|S_5|\\
\lambda_4 &=\sum_{s\in S}\chi_4 (s)=-|S_2|+|S_5|
\end{align*}
So $\lambda_3=\lambda_4$. Since $\Gamma$ is a connected $|S|$-regular graph, it follows from \cite[Proposition 1.1.2]{CRS} that $\mu_1:=\lambda_1=|S|$ is the unique largest eigenvalue of $\Gamma$. Thus $1$ occurs precisely one time in the multiplicity type of spectrum of $\Gamma$.
By Theorem \ref{Babai} there is a multiset $\Lambda:=\{\lambda_5,\dots,\lambda_{20}\}$ of $16$ eigenvalues of $\Gamma$ corresponding to the character $\psi$ of degree $4$ such that the multiplicity of each $\lambda\in \Lambda$ in $\Lambda$ is $4$, $8$, $12$ or $16$; note that the latter multiplicity is counted in $\Lambda$ and not in the spectrum of $\Gamma$.
If $\lambda_2\neq \lambda_3$, it follows from the above information that  $\Gamma$ has the spectrum of the first type; otherwise the spectrum of $\Gamma$ is of the second type. This completes the proof.
\end{proof}

\begin{proof}[Proof of Theorem \ref{F5}]
It follows from \cite[Theorem 1.1]{LLP} that the group $\mathcal{G}$ is not a CI-group.\\
To prove that $\mathcal{G}$ is a BI-group, by Proposition \ref{red-BI} it is enough to show that $M_\nu^S=M_\nu^T$ for every irreducible character degree $\nu$ of $\mathcal{G}$ and for all inverse closed subsets $S$ and $T$ of $\mathcal{G}^*$ such that  $Cay(\mathcal{G},S)\cong Cay(\mathcal{G},T)$ and $\mathcal{G}=\langle S\rangle=\langle T\rangle$.\\ 
First note that $\nu\in\{1,4\}$ and  $$M_1 ^S=\{ |S|, |S_2|-|S_4|+|S_5|, |S_5|-|S_2| \}, M_4^S= \{-|S_5| \} $$
and similarly $$M_1 ^T=\{ |T|, |T_2|-|T_4|+|T_5|, |T_5|-|T_2| \}, M_4^T= \{-|T_5| \}.$$
Since $Cay(\mathcal{G},S)\cong Cay(\mathcal{G},T)$, $spec(Cay(\mathcal{G},S))=spec(Cay(\mathcal{G},T))$. By Theorem \ref{g20}, it follows that 
\begin{displaymath}
\left\{ \begin{aligned}
&|S_2|+|S_4|+|S_5|=|S|=|T|=|T_2|+|T_4|+|T_5|\\
&|S_2|-|S_4|+|S_5|=|T_2|-|T_4|+|T_5|\\
& -|S_2|+|S_5| =-|T_2|+|T_5|
\end{aligned}.\right.
\end{displaymath}
Therefore, 
$$|S_2|=|T_2|,|S_4|=|T_4|,|S_5|=|T_5|,$$
and so $M_\nu ^S=M_\nu ^T$ for all $\nu$. This completes the proof.
\end{proof}

\section{\bf Proof of  Theorem \ref{F7}}

\begin{thm}\label{g42}
Let $\mathcal{H}=\langle a,b|a^7=b^6=1, a^b=a^3\rangle$, $S\subset \mathcal{H}^*$, $S^{-1}=S$, $\mathcal{H}=\langle S \rangle$ and $\Gamma=Cay(\mathcal{H},S)$. Then the spectrum of $\Gamma$ is one of the following four cases:
\begin{enumerate}
\item $spec(\Gamma)=(\mu_1^1=|S|,\mu_2^{6k+1},\mu_3^{6k'+2},\mu_4^{6k''+2},\gamma_1^{6l_1},\dots,\gamma_n^{6l_n})$,\\
where
$\mu_2=-|S_2|+|S_{3_1}|+|S_{3_2}|-|S_{6_1}|-|S_{6_2}|+|S_7|$, $\mu_3=|S_2|-|S_{3_1}|-|S_{6_1}|+|S_7|$,
$\mu_4=-|S_2|-|S_{3_1}|+|S_{6_1}|+|S_7|$ and $\mu_2$, $\mu_3$ and $\mu_4$ are pairwise distinct;
\item $spec(\Gamma)=(\mu_1^1=|S|,\mu_2^{6k+1},\mu_3^{6k'+4},\gamma_1^{6l_1},\dots,\gamma_n^{6l_n})$,\\
where
$\mu_2=-|S_2|+|S_{3_1}|+|S_{3_2}|-|S_{6_1}|-|S_{6_2}|+|S_7|$, $\mu_3=|S_2|-|S_{3_1}|-|S_{6_1}|+|S_7|=-|S_2|-|S_{3_1}|+|S_{6_1}|+|S_7|$, $\mu_2\neq \mu_3$;
\item $spec(\Gamma)=(\mu_1^1=|S|,\mu_2^{6k+3},\mu_3^{6k'+2},\gamma_1^{6l_1},\dots,\gamma_n^{6l_n})$,\\
where
$\mu_2=-|S_2|+|S_{3_1}|+|S_{3_2}|-|S_{6_1}|-|S_{6_2}|+|S_7|=-|S_2|-|S_{3_1}|+|S_{6_1}|+|S_7|$
and $\mu_3=|S_2|-|S_{3_1}|-|S_{6_1}|+|S_7|$, $\mu_2 \neq \mu_3$,
\item $spec(\Gamma)=(\mu_1^1=|S|,\mu_2^{6k+3},\mu_3^{6k'+2},\gamma_1^{6l_1},\dots,\gamma_n^{6l_n})$,\\
where
$\mu_2=-|S_2|+|S_{3_1}|+|S_{3_2}|-|S_{6_1}|-|S_{6_2}|+|S_7|=|S_2|-|S_{3_1}|-|S_{6_1}|+|S_7|$
and $\mu_3=-|S_2|-|S_{3_1}|+|S_{6_1}|+|S_7|$, $\mu_2 \neq \mu_3$;
\item $spec(\Gamma)=(\mu_1^1=|S|,\mu_2^{6k+5},\gamma_1^{6l_1},\dots,\gamma_n^{6l_n})$,
where
$\mu_2=-|S_2|+|S_{3_1}|+|S_{3_2}|-|S_{6_1}|-|S_{6_2}|+|S_7|=|S_2|-|S_{3_1}|-|S_{6_1}|+|S_7|=-|S_2|-|S_{3_1}|+|S_{6_1}|+|S_7|$;
\end{enumerate}
where
$k,k',k'',l_{i}$ and $n$ are integers in $\{0,\dots,6\}$.
\end{thm}

\begin{proof}
The group $\mathcal{H}$ has $7$ conjugacy classes: $\mathcal{H}_1$, $\mathcal{H}_2$ and $\mathcal{H}_7$ are the conjugacy classes consisting of all elements of $\mathcal{H}$ of order $1$, $2$ and $7$, respectively; the elements of order $i\in\{3,6\}$ of $\mathcal{H}$ are partitioned into $2$ conjugacy classes denoted by $\mathcal{H}_{i_1}$ and $\mathcal{H}_{i_2}$  and if $x$ is an element of order $i$, then $x$ and $x^{-1}$ are not conjugate.
Let $S_{i_j}=S \cap \mathcal{H}_{i_j}$ for $i=3,6$ and $j=1,2$; so $S_{i_1}^{-1}=S_{i_2}$ for $i=3,6$. According to Definition \ref{spe}:
$$|S|=|S_2|+|S_{3_1}|+|S_{3_2}|+|S_{6_1}|+|S_{6_2}|+|S_{7}|,$$
and note that $S_3=S_{3_1}\cup S_{3_2}$, $|S_{3_1}|=|S_{3_2}|$, $S_6=S_{6_1}\cup S_{6_2}$ and $|S_{6_1}|=|S_{6_2}|$.
The character table of $\mathcal{H}$ is:
\begin{table}[H]
\centering
\begin{tabular}{c||ccccccc} 
Class&$\mathcal{H}_1$&$\mathcal{H}_2$&$\mathcal{H}_{3_1}$&$\mathcal{H}_{3_2}$&$\mathcal{H}_{6_1}$&$\mathcal{H}_{6_2}$&$\mathcal{H}_7$\\\hline
Size&$1$&$7$&$7$&$7$&$7$&$7$&$6$\\\hline
$\chi_1$&$1$&$1$&$1$&$1$&$1$&$1$&$1$\\
$\chi_2$&$1$&$-1$&$1$&$1$&$-1$&$-1$&$1$\\
$\chi_3$&$1$&$1$&$\zeta_3^2$&$\zeta_3$&$\zeta_3$&$\zeta_3^2$&$1$\\
$\chi_4$&$1$&$1$&$\zeta_3$&$\zeta_3^2$&$\zeta_3^2$&$\zeta_3$&$1$\\
$\chi_5$&$1$&$-1$&$\zeta_3$&$\zeta_3^2$&$\zeta_6$&$\zeta_6^5$&$1$\\
$\chi_6$&$1$&$-1$&$\zeta_3^2$&$\zeta_3$&$\zeta_6^5$&$\zeta_6$&$1$\\
$\psi$&$6$&$0$&$0$&$0$&$0$&$0$&$-1$\\
\end{tabular}
\caption{\label{table:F7}\small{Character table of the group $\mathcal{H}$}}
\end{table}
where $\zeta_l=e^{2\pi {\mathbf i}/l}$ ($l=3,6$). 

Now by Theorem \ref{Babai}, all eigenvalues of $\Gamma$ corresponding to irreducible characters $\chi_1,\dots,\chi_6$ of $\mathcal{H}$ are as follows:
\begin{eqnarray*}
&&\lambda_1 =\sum_{s\in S}\chi_1 (s)=|S|=|S_2|+|S_{3_1}|+|S_{3_2}|+|S_{6_1}|+|S_{6_2}|+|S_7|\\
&&\lambda_2 =\sum_{s\in S}\chi_2 (s)=-|S_2|+|S_{3_1}|+|S_{3_2}|-|S_{6_1}|-|S_{6_2}|+|S_7|\\
&&\lambda_3=\sum_{s\in S}\chi_3 (s)=|S_2|-|S_{3_1}|-|S_{6_1}|+|S_7|\\
&&\lambda_4=\sum_{s\in S}\chi_4 (s)=|S_2|-|S_{3_1}|-|S_{6_1}|+|S_7|\\
&&\lambda_5=\sum_{s\in S}\chi_5 (s)=-|S_2|-|S_{3_1}|+|S_{6_1}|+|S_7|\\
&&\lambda_6=\sum_{s\in S}\chi_6 (s)=-|S_2|-|S_{3_1}|+|S_{6_1}|+|S_7|\\
\end{eqnarray*}
So $\lambda_3=\lambda_4$ and $\lambda_5=\lambda_6$. Since $\Gamma$ is a connected $|S|$-regular graph, it follows from \cite[Proposition 1.1.2]{CRS} that $\mu_1:=\lambda_1=|S|$ is the unique largest eigenvalue of $\Gamma$. 

By Theorem \ref{Babai} there exists a multiset $\Lambda:=\{\lambda_7,\dots,\lambda_{42}\}$ of $36$ eigenvalues of $\Gamma$ corresponding to the character $\psi$ of degree $6$ such that the multiplicity of each $\lambda\in \Lambda$ in $\Lambda$ is a multiple of  $6$; note that the latter multiplicity is counted in $\Lambda$ and not in the spectrum of $\Gamma$. 



Now, we must distinguish  five cases to obtain the multiplicities of  eigenvalues $\lambda_2,\lambda_3,\lambda_4,\lambda_5,\lambda_6$:

\item[\bf Case 1.]
$\lambda_2$, $\lambda_3$ and $\lambda_5$ are pairwise distinct. Then $\Gamma$ has the spectrum of the first type $(1)$. 
\item[\bf Case 2.]
$\lambda_2\neq \lambda_3= \lambda_5$. Then $\Gamma$ has the spectrum of the second type $(2)$.
\item[\bf Case 3.]
$\lambda_2= \lambda_3\neq \lambda_5$. Then $\Gamma$ has the spectrum of the third type $(3)$.
\item[\bf Case 4.]
$\lambda_2= \lambda_5\neq \lambda_3$. Then $\Gamma$ has the spectrum of the third type $(4)$.

\item[\bf Case 5.]
$\lambda_2= \lambda_3= \lambda_5$. 
Then $\Gamma$ has the spectrum of the fourth type $(5)$.

\noindent This completes the proof.
\end{proof}

\begin{proof}[Proof of Theorem \ref{F7}]
It follows from \cite[Theorem 1.1]{LLP} that the group $\mathcal{H}$ is not a CI-group.\\
To prove that $\mathcal{H}$ is a BI-group, by Proposition \ref{red-BI} it is enough to show that $M_\nu^S=M_\nu^T$ for every irreducible character degree $\nu$ of $\mathcal{H}$ and for all inverse closed subsets $S$ and $T$ of $\mathcal{H}^*$ such that  $\Gamma:=Cay(\mathcal{H},S)\cong Cay(\mathcal{H},T)=:\Gamma'$ and $\mathcal{H}=\langle S\rangle=\langle T\rangle$.\\ 
First note that $\nu\in\{1,6\}$; and by the proof of Theorem \ref{g42}
$$M_1 ^S\!=\!\{|S|, -|S_2|+|S_3|-|S_6|+|S_7|, |S_2|-\frac{1}{2}|S_3|-\frac{1}{2}|S_6|+|S_7|,-|S_2|-\frac{1}{2}|S_3|+\frac{1}{2}|S_6|+|S_7| \}$$
and
$M_6^S= \{-|S_7| \}$
and similarly
$$M_1 ^T\!=\!\{|T|, -|T_2|+|T_3|-|T_6|+|T_7|, |T_2|-\frac{1}{2}|T_3|-\frac{1}{2}|T_6|+|T_7|,-|T_2|-\frac{1}{2}|T_3|+\frac{1}{2}|T_6|+|T_7| \}$$
and $M_6^T= \{-|T_7| \}$. 
Since $\Gamma\cong \Gamma'$, $spec(\Gamma)=spec(\Gamma')$. 
By  Theorem \ref{g42}, 
the following cases may happen:
\begin{itemize}
\item[(i)] $\Gamma$ and $\Gamma'$ have the same spectrum as type (1) in Theorem \ref{g42};
\item[(ii)] $\Gamma$ has and $\Gamma'$ have the same spectrum as type (2);
\item[(iii)]  $\Gamma$ and $\Gamma'$ have the same spectrum as type (3);
\item[(iv)] $\Gamma$ and $\Gamma'$ have the same spectrum as type (4);
\item[(v)] $\Gamma$ and $\Gamma'$ have the same spectrum as type (5);
\item[(vi)] $spec(\Gamma)$ is of type (3)  and $spec(\Gamma')$ is of type (4);
\end{itemize}

By Theorem \ref{g42}, $M_1^S$ is the set of eigenvalues of $\Gamma$ having the multiplicity $6k+\ell$ for some non-negative integer $k$ and positive integer $\ell\leq 5$. The similar statement is true for $M_1^T$ and so $M_1^S=M_1^T$ in all cases (i)-(vi).

Now it remains to prove that $M_6^S=M_6^T$ or equivalently $|S_7|=|T_7|$.\\
 With the notations of Theorem \ref{g42}, let $\mu_i:=\mu_i(\Gamma)$ and $\mu_i':=\mu_i(\Gamma')$. 

According to the above cases (i)-(vi), we distinguish the following cases: 

(i) \; $\Gamma$ and $\Gamma'$ have the same spectrum as type (1) in Theorem \ref{g42}; 
\begin{displaymath}
(I) \left\{ \begin{aligned}
&\mu_1=|S|=|S_2|+|S_3|+|S_6|+|S_7|=|T_2|+|T_3|+|T_6|+|T_7|=|T|=\mu'_1\\
&\mu_2=-|S_2|+|S_3|-|S_6|+|S_7|=-|T_2|+|T_3|-|T_6|+|T_7|=\mu'_2\\
&\mu_3=|S_2|-\frac{1}{2}|S_3|-\frac{1}{2}|S_6|+|S_7| = |T_2|-\frac{1}{2}|T_3|-\frac{1}{2}|T_6|+|T_7|=\mu'_3\\
&\mu_5=-|S_2|-\frac{1}{2}|S_3|+\frac{1}{2}|S_6|+|S_7|=-|T_2|-\frac{1}{2}|T_3|+\frac{1}{2}|T_6|+|T_7|=\mu'_5
\end{aligned},\right.
\end{displaymath}
or
\begin{displaymath}
(II) \left\{ \begin{aligned}
&\mu_1=|S|=|S_2|+|S_3|+|S_6|+|S_7|=|T_2|+|T_3|+|T_6|+|T_7|=|T|=\mu'_1\\
&\mu_2=-|S_2|+|S_3|-|S_6|+|S_7|=-|T_2|+|T_3|-|T_6|+|T_7|=\mu'_2\\
&\mu_3=|S_2|-\frac{1}{2}|S_3|-\frac{1}{2}|S_6|+|S_7| =-|T_2|-\frac{1}{2}|T_3|+\frac{1}{2}|T_6|+|T_7|=\mu'_5\\
&\mu_5=-|S_2|-\frac{1}{2}|S_3|+\frac{1}{2}|S_6|+|S_7|=|T_2|-\frac{1}{2}|T_3|-\frac{1}{2}|T_6|+|T_7|=\mu'_3
\end{aligned}.\right.
\end{displaymath}

If (I) happens, then as the matrix 
$\begin{bmatrix}
1  & 1 & 1 &1 \\
-1 & 1  & -1 & 1 \\
1 & -\frac{1}{2} & -\frac{1}{2} & 1 \\
-1 & -\frac{1}{2} & \frac{1}{2} &1 
\end{bmatrix}$ is invertible then $|S_7|=|T_7|$.  

If (II) happens then
$$\begin{bmatrix}
1  & 1 & 1 & 1 \\
-1 & 1  & -1 & 1 \\
1 & -1/2 & -1/2 & 1 \\
-1 & -1/2 & 1/2 & 1 
\end{bmatrix}  \begin{bmatrix} |S_2| \\ |S_3| \\ |S_6| \\ |S_7| \end{bmatrix} = \begin{bmatrix} 
1  & 1 & 1 & 1 \\
-1 & 1  & -1 & 1 \\
-1 & -1/2 & 1/2 & 1 \\
1 & -1/2 & -1/2 & 1 
\end{bmatrix}  \begin{bmatrix} |T_2| \\ |T_3| \\ |T_6| \\ |T_7| \end{bmatrix} $$
so
$$\begin{bmatrix} |S_2| \\ |S_3| \\ |S_6| \\ |S_7| \end{bmatrix}=
\begin{bmatrix}
-1/3  & 0 & 2/3 & 0 \\
0 & 1  & 0 & 0 \\
4/3 & 0 & 1/3 & 0 \\
0 & 0 & 0 & 1 
\end{bmatrix}
\begin{bmatrix} |T_2| \\ |T_3| \\ |T_6| \\ |T_7| \end{bmatrix}.$$

(ii) \; $\Gamma$ and $\Gamma'$ have the same spectrum as type (2) in Theorem \ref{g42};

\begin{displaymath}
\left\{ \begin{aligned}
&\mu_1=|S|=|S_2|+|S_3|+|S_6|+|S_7|=|T_2|+|T_3|+|T_6|+|T_7|=|T|=\mu'_1\\
&\mu_2=-|S_2|+|S_3|-|S_6|+|S_7|=-|T_2|+|T_3|-|T_6|+|T_7|=\mu'_2\\
&\mu_3=|S_2|-\frac{1}{2}|S_3|-\frac{1}{2}|S_6|+|S_7| = |T_2|-\frac{1}{2}|T_3|-\frac{1}{2}|T_6|+|T_7|=\mu'_3\\
&\mu_5=-|S_2|-\frac{1}{2}|S_3|+\frac{1}{2}|S_6|+|S_7|=-|T_2|-\frac{1}{2}|T_3|+\frac{1}{2}|T_6|+|T_7|=\mu'_5\\
&\mu_3=\mu_5
\end{aligned}.\right.
\end{displaymath}

(iii) \; $\Gamma$ and $\Gamma'$ have the same spectrum as type (3) in Theorem \ref{g42};

\begin{displaymath}
\left\{ \begin{aligned}
&\mu_1=|S|=|S_2|+|S_3|+|S_6|+|S_7|=|T_2|+|T_3|+|T_6|+|T_7|=|T|=\mu'_1\\
&\mu_2=-|S_2|+|S_3|-|S_6|+|S_7|=-|T_2|+|T_3|-|T_6|+|T_7|=\mu'_2\\
&\mu_3=|S_2|-\frac{1}{2}|S_3|-\frac{1}{2}|S_6|+|S_7| = |T_2|-\frac{1}{2}|T_3|-\frac{1}{2}|T_6|+|T_7|=\mu'_3\\
&\mu_5=-|S_2|-\frac{1}{2}|S_3|+\frac{1}{2}|S_6|+|S_7|=-|T_2|-\frac{1}{2}|T_3|+\frac{1}{2}|T_6|+|T_7|=\mu'_5\\
&\mu_2=\mu_3
\end{aligned},\right.
\end{displaymath}

(iv) \; $\Gamma$ and $\Gamma'$ have the same spectrum as type (4) in Theorem \ref{g42};

\begin{displaymath}
\left\{ \begin{aligned}
&\mu_1=|S|=|S_2|+|S_3|+|S_6|+|S_7|=|T_2|+|T_3|+|T_6|+|T_7|=|T|=\mu'_1\\
&\mu_2=-|S_2|+|S_3|-|S_6|+|S_7|=-|T_2|+|T_3|-|T_6|+|T_7|=\mu'_2\\
&\mu_3=|S_2|-\frac{1}{2}|S_3|-\frac{1}{2}|S_6|+|S_7| = |T_2|-\frac{1}{2}|T_3|-\frac{1}{2}|T_6|+|T_7|=\mu'_3\\
&\mu_5=-|S_2|-\frac{1}{2}|S_3|+\frac{1}{2}|S_6|+|S_7|=-|T_2|-\frac{1}{2}|T_3|+\frac{1}{2}|T_6|+|T_7|=\mu'_5\\
&\mu_2=\mu_5
\end{aligned},\right.
\end{displaymath}

v) \; $\Gamma$ and $\Gamma'$ have the same spectrum as type (5) in Theorem \ref{g42};

\begin{displaymath}
\left\{ \begin{aligned}
&\mu_1=|S|=|S_2|+|S_3|+|S_6|+|S_7|=|T_2|+|T_3|+|T_6|+|T_7|=|T|=\mu'_1\\
&\mu_2=-|S_2|+|S_3|-|S_6|+|S_7|=-|T_2|+|T_3|-|T_6|+|T_7|=\mu'_2\\
&\mu_3=|S_2|-\frac{1}{2}|S_3|-\frac{1}{2}|S_6|+|S_7| = |T_2|-\frac{1}{2}|T_3|-\frac{1}{2}|T_6|+|T_7|=\mu'_3\\
&\mu_5=-|S_2|-\frac{1}{2}|S_3|+\frac{1}{2}|S_6|+|S_7|=-|T_2|-\frac{1}{2}|T_3|+\frac{1}{2}|T_6|+|T_7|=\mu'_5\\
&\mu_2=\mu_3=\mu_5
\end{aligned}.\right.
\end{displaymath}

(vi) \; $\Gamma$ has the spectrum as type (3) and $\Gamma'$ has the same spectrum but as type (4) in Theorem \ref{g42}, conversely is true if $\Gamma$ has the spectrum as type (4) and $\Gamma'$ has the same spectrum but as type (3);

\begin{displaymath}
\left\{ \begin{aligned}
&\mu_1=|S|=|S_2|+|S_3|+|S_6|+|S_7|=|T_2|+|T_3|+|T_6|+|T_7|=|T|=\mu'_1\\
&\mu_2=-|S_2|+|S_3|-|S_6|+|S_7|=-|T_2|+|T_3|-|T_6|+|T_7|=\mu'_2\\
&\mu_3=|S_2|-\frac{1}{2}|S_3|-\frac{1}{2}|S_6|+|S_7| =-|T_2|-\frac{1}{2}|T_3|+\frac{1}{2}|T_6|+|T_7|=\mu'_5\\
&\mu_5=-|S_2|-\frac{1}{2}|S_3|+\frac{1}{2}|S_6|+|S_7|=|T_2|-\frac{1}{2}|T_3|-\frac{1}{2}|T_6|+|T_7|=\mu'_3\\
&\mu_2=\mu_3 \text{ or } \mu_2=\mu_5
\end{aligned}.\right.
\end{displaymath}
Then
$$\begin{bmatrix}
1  & 1 & 1 & 1 \\
-1 & 1  & -1 & 1 \\
1 & -1/2 & -1/2 & 1 \\
-1 & -1/2 & 1/2 & 1
\end{bmatrix}  \begin{bmatrix} |S_2| \\ |S_3| \\ |S_6| \\ |S_7| \end{bmatrix} = \begin{bmatrix} 
1  & 1 & 1 & 1 \\
-1 & 1  & -1 & 1 \\
-1 & -1/2 & 1/2 & 1 \\
1 & -1/2 & -1/2 & 1 
\end{bmatrix}  \begin{bmatrix} |T_2| \\ |T_3| \\ |T_6| \\ |T_7| \end{bmatrix} $$
so
$$\begin{bmatrix} |S_2| \\ |S_3| \\ |S_6| \\ |S_7| \end{bmatrix}=
\begin{bmatrix}
-1/3  & 0 & 2/3 & 0 \\
0 & 1  & 0 & 0 \\
4/3 & 0 & 1/3 & 0 \\
0 & 0 & 0 & 1 
\end{bmatrix}
\begin{bmatrix} |T_2| \\ |T_3| \\ |T_6| \\ |T_7| \end{bmatrix}.$$

Therefore, in all of cases one may conclude that  $|S_7|=|T_7|$; this completes the proof.
\end{proof}

\begin{rem}
It can be proved that the cases (i)-II and (vi) will not happen in the proof of Theorem \ref{F7}.  
\end{rem}

\section{\bf Proof of Theorem \ref{BI-groups}}
We have used the following results to prove that some of groups of order at most 30 is not a BI-group.
\begin{thm}\label{nonBI}
Let $G$ be a non-abelian group. Suppose that there exist elements $h\in G\setminus G'$ and $k\in G'$ of the same order. Assume further that $\lambda(h)=\lambda(h^{-1})$ for all linear characters $\lambda$ of $G$ (e.g.  $h$ is conjugate to $h^{-1}$ or all linear characters of $G$ are real on $h$). Then $G$ is not a BI-group.
\end{thm}
\begin{proof}
Let
$$S=\{h,h^{-1}\} \text{ and } T=\{k,k^{-1}\},$$
Since $o(h)=o(k)$, $Cay(G,S)\cong Cay(G,T)$. Then
$$M_1 ^T=\left\{ \begin{array}{rcl}
\{1\} & o(k)=2\\
\{2\} & o(k)> 2
\end{array},\right.
$$
since $T\subset G'=\bigcap\{ker(\lambda)| \lambda \in Irr(G), \lambda(1)=1\}$ by \cite[Corollary 2.23]{Isaacs}. On the other hand $S\not\subset G'$ implies that there exists a linear character $\lambda\in Irr(G)$ such that $\lambda(h)\neq 1$.
It follows that $M^S_1\neq M^T_1$. This completes the proof.
\end{proof}

\begin{thm}\label{sem}
Let $G=H\ltimes K$, $H$ is an abelian group and $K$ is a non-abelian group. Assume that non-trivial elements $h\in H$ and $k\in K'$ have the same order and all linear characters of $G$ has real value on $h$. Then $G$ is not a BI-group.
\end{thm}
\begin{proof}
Note that $K'\subset G'$ and $G'\cap H=1$.
By the proof of Theorem \ref{nonBI} for
$S=\{(h,1),(h,1)^{-1}\}$ and $T=\{(1,k),(1,k)^{-1}\}$ we have
$$M_1 ^T=\left\{ \begin{array}{rcl}
\{1\} & o(k)=2\\
\{2\} & o(k)> 2
\end{array},\right.
$$
and
$$M_1 ^S=\left\{ \begin{array}{rcl}
\{\lambda(h,1)|\lambda \in Irr(G), \lambda(1)=1\} & o(h)=2\\
\{2\lambda(h,1)|\lambda \in Irr(G), \lambda(1)=1\} & o(h)> 2
\end{array},\right.
$$
since $S\not\subset G'$ and $T\subset G'$ and there is a linear character of $G$ which is real on $h$. So, $M^S_1\neq M^T_1$. This completes the proof.
\end{proof}

\begin{thm}\label{di}
Let $G=H\times K$, $H$ is an abelian group and $K$ is a non-abelian group. Assume that non-trivial elements $h\in H$ and $k\in K'$ have the same order and all irreducible characters of $H$ have real value on $h$. Then $G$ is not a BI-group.
\end{thm}
\begin{proof}
Since the irreducible characters of $H$ on $h$ have the same values to corresponding irreducible characters of $G$ on $(h,1)$, the proof is similar to the proof of Theorem \ref{sem}.
\end{proof}

\begin{prop}\label{SL}
The group $SL(2,3)$ is not a BI-group.
\end{prop}
\begin{proof}
Let $S$ and $T$ as the following:

$ S=\left\{\begin{gmatrix}[q]
  1 & 1\\  0 & 1
\end{gmatrix},
\begin{gmatrix}[q]
  0 & 1\\  2 & 2
\end{gmatrix},
\begin{gmatrix}[q]
  0 & 2\\  1 & 2
\end{gmatrix},
\begin{gmatrix}[q]
  2 & 1\\  2 & 0
\end{gmatrix},
\begin{gmatrix}[q]
  1 & 2\\  0 & 1
\end{gmatrix},
\begin{gmatrix}[q]
  2 & 2\\  1 & 0
  \end{gmatrix}
 \right\},$

$T=\left\{\begin{gmatrix}[q]
  2 & 0\\  1 & 2
\end{gmatrix},
\begin{gmatrix}[q]
  1 & 2\\  1 & 0
\end{gmatrix},
\begin{gmatrix}[q]
  1 & 1\\  2 & 0
\end{gmatrix},
\begin{gmatrix}[q]
  0 & 1\\  2 & 1
\end{gmatrix},
\begin{gmatrix}[q]
  0 & 2\\  1 & 1
\end{gmatrix},
\begin{gmatrix}[q]
  2 & 0\\  2 & 2
  \end{gmatrix}
 \right\}.$

According to the character table of $G:=SL(2,3)$ it is easy to see $M_2^S=-M_2^T$, while $Cay(G,S)\cong Cay(G,T)$, by using GAP \cite{GAP}. So $SL(2,3)$ is not a BI-group.
\end{proof}

\begin{thm}\cite{BF}\label{Sylow}
Let $P$ be a Sylow $p$-subgroup of a CI-group. Then $P$ is either an elementary abelian group, or a cyclic group of order $p^l$, $(l,p\leq 3)$, or the quaternion group $Q_8$.
\end{thm}

\begin{thm}\cite[Proposition 3.5]{AM}\label{dihedral}
Dihedral group $D_{2n}$ of order $2n$ is not a BI-group for even $n$.
\end{thm}

\begin{prop}\label{2-groups}
All non-abelian groups of order $16$ are not BI-groups.
\end{prop}
\begin{proof}
It follows from to Theorems \ref{di}, \ref{sem} and \ref{dihedral}.
\end{proof}

\begin{proof}[Proof of Theorem \ref{BI-groups}]
The proof follows from  the results which are mentioned in  Table \ref{table:3}.
\end{proof}

\small{
\begin{table}
\centering
\caption{\label{table:3}\small{Non-abelian CI and BI groups of order at most $30$. A reference is given for each item to establish if the group is CI or BI or not. Y. means Yes; N. means No.  The reference \cite{GAP} shows that we have used the above code written in GAP to determine the group is BI or not. }}

\begin{tabular}{|c|c|c|c||c|c|c|c|}
\hline
ID&Group&CI&BI&ID&Group&CI&BI\\
\hline
[6,1]&$S_3$&Y.\cite{B2}&Y.\cite{AM}&
[8,3]&$D_8$&N.\cite{Li2}&N.\ref{dihedral}\\\hline
[8,4]&$Q_8$&Y.\cite{Li2}&Y.\cite{AM}&
[10,1]&$D_{10}$&Y.\cite{B2}&Y.\cite{AM}\\\hline
[12,1]&$C_4\ltimes C_3$&Y.\cite{Roy}&Y.\cite{AM}&
[12,3]&$A_4$&Y.\cite{Roy}&Y.\ref{BICI}\\\hline
[12,4]&$D_{12}$&N.\cite{Roy}&N.\ref{dihedral}&
[14,1]&$D_{14}$&Y.\cite{B2}&Y.\cite{AM}\\\hline
[16,-]&\multicolumn{4}{l}
\mbox{All non-abelian groups of order 16}&&N. \ref{2-groups}&N. \ref{2-groups}\\\hline
[18,1]&$D_{18}$&Y.\cite{Roy}&Y.\ref{BICI}&
[18,3]&$C_3\times S_3$&N.\cite{Roy}&N.\ref{di}\\\hline
[18,4]&$C_2\ltimes (C_3\times C_3)$&Y.\cite{Roy}&Y.\ref{BICI}&
[20,1]&$C_4\ltimes C_5$&Y.\cite{LLP}&Y.\cite{AM}\\\hline
[20,3]&$C_4\ltimes C_5$&N.\cite{LLP}&Y.\ref{F5}&
[20,4]&$D_{20}$&N.\cite{Roy}&N.\ref{dihedral}\\\hline
[21,1]&$C_3\ltimes C_7$&Y.\cite{Dobson}&Y.\ref{BICI}&
[22,1]&$D_{22}$&Y.\cite{B2}&Y.\cite{AM}\\\hline
[24,1]&$C_8\ltimes C_3$&Y.\cite{Roy}&Y.\ref{BICI}&
[24,3]&$SL(2,3)$&N.\cite{Conder-Li}&N.\ref{SL}\\\hline
[24,4]&$C_8\ltimes C_3$&N.\cite{Conder-Li}&N.\cite{GAP}&
[24,5]&$C_4\times S_3$&N.\cite{Conder-Li}&N.\ref{sem}\\\hline
[24,6]&$D_{24}$&N.\cite{Roy}&N.\ref{dihedral}&
[24,7]&$C_2\times (C_4\ltimes C_3)$&N.\ref{Sylow}&N.\cite{GAP}\\\hline
[24,8]&$C_2\ltimes (C_2 \times C_6)$&N.\ref{Sylow}&N.\cite{GAP}&
[24,10]&$C_3\times D_4$&N.\ref{Sylow}&N.\cite{GAP}\\\hline
[24,11]&$C_3\times Q_8$&N.\cite{Spiga}&N.\cite{GAP}&
[24,12]&$S_4$&N.\ref{Sylow}&N.\cite{AM}\\\hline
[24,13]&$C_2\times A_4$&N.\ref{di}&N.\ref{di}&
[24,14]&$C_2\times C_2 \times S_3$&N.\ref{di}&N.\ref{di}\\\hline
[26,1]&$D_{26}$&Y.\cite{B2}&Y.\cite{AM}&
[27,3]&$C_3\ltimes (C_3\times C_3)$&N.\cite{AM}&N.\ref{sem}\\\hline
[27,4]&$C_3\ltimes C_9$&N.\cite{AM}&N.\ref{sem}&
[28,1]&$C_4\ltimes C_7$&Y.\cite{LLP}&Y.\cite{AM}\\\hline
[28,3]&$D_{28}$&N.\cite{Roy}&N.\ref{dihedral}&
[30,1]&$C_5\times S_3$&Y.\cite{Roy}&Y.\ref{BICI}\\\hline
[30,2]&$C_3\times D_{10}$&Y.\cite{DobMorSpi}&Y.\ref{BICI}&
[30,3]&$D_{30}$&Y.\cite{Roy}&Y.\ref{BICI}\\\hline
\end{tabular}\\
\end{table}
}
\section*{\bf Acknowledgements}
The research of the first author was in part supported by a grant (No. 96050219) from School of Mathematics, Institute
for Research in Fundamental Sciences (IPM).


\begin{thebibliography}{99}
\bibitem{AM} A. Abdollahi and M. Zallaghi, Character sums for Cayley graphs, \textit{Comm. Algebra}, {\bf  43} (12) (2015) 5159-5167.
\bibitem{Als} B. Alspach, Isomorphism and Cayley graphs on abelian groups, \textit{NATO Sci. Peace Secur. Ser. C Graph Symm.}, {\bf 497} (1997) 1-22.
\bibitem{B} L. Babai, Spectra of Cayley graphs, \textit{J. Combin. Theory Ser. B}, {\bf  27} (1979) 180-189.
\bibitem{B2} L. Babai, Isomorphism problem for a class of point-symmetric structures, \textit{Acta Math. Hungar.}, {\bf  29} (1977) 329-336.
\bibitem{BF} L. Babai and P. Frankl, Isomorphisms of Cayley graphs I, \textit{Colloq. Math. Soc. J. Bolyai}, {\bf  18} (1976) 35-52.
\bibitem{Conder-Li} M. Conder and C. H. Li, On isomorphisms of finite Cayley graphs, \textit{European J. Combin.}, {\bf 19} (1998) 911-919.
\bibitem{CRS} D. Cvetkovi\'c, P. Rowlinson and S. Simi\'c, An introduction to the theory of graph spectra, Cambridge, UK: Cambridge University Press, 2010.
\bibitem{Dobson} E. Dobson, Isomorphism problem for metacirculant graphs of order a product of distinct primes, \textit{Canad. J. Math.}, {\bf 50} (1998) 1176-1188.
\bibitem{Dobson1} E. Dobson, On the Cayley isomorphism problem, \textit{Discrete Math.}, {\bf 247} (2002) 107-116.
\bibitem{DobMorSpi} E. Dobson, J. Morris and P. Spiga, Further restrictions on the structure of finite DCI-groups: an addendum, \textit{J. Algebraic Combin.}, {\bf 42} (2015) 959-969.
\bibitem{Isaacs} I. Martin Isaacs, Character theory of finite groups, Acad. Press, New York, 1976.
\bibitem{Li2} C. H. Li, On isomorphisms of finite Cayley graphs-a survey, \textit{Discrete Math.}, {\bf 256} (2002) 301-334.
\bibitem{LLP} C. H. Li, Z. P. Lu and P. P. P\'{a}lfy, Further restrictions on the structure of finite CI-groups, \textit{J. Algebraic Combin.}, {\bf 26} (2007) 161-181.
\bibitem{Lov} L. Lov\'asz, Spectra of graphs with transitive groups, \textit{Period. Math. Hungar.}, {\bf 6} (1975) 191-196.
\bibitem{Roy} G. F. Royle, Constructive enumeration of graphs, Ph.D. thesis, \textit{University of Western Australia}, 1987.
\bibitem{So1} G. Somlai, The Cayley isomorphism property for groups of order $8p$,  \textit{Ars Math. Contemp.}, {\bf 8} (2015) 433-444.
\bibitem{Spiga} P‎. ‎Spiga‎, ‎On the Cayley isomorphism problem for a digraph with 24 vertices‎, \textit{Ars Math. Contemp.}, {\bf 1} (2008) no‎. ‎1‎, ‎38-43‎.
\bibitem{GAP} The GAP Group, GAP Groups, Algorithms, and Programming, Version 4.4.12 (2008), (\href{http://www.gap-system.org}{http://www.gap-system.org}).
\end{thebibliography}
\end{document}